\documentclass[10pt]{amsart}

\usepackage{amsmath,amssymb,setspace,nicefrac}
\usepackage[active]{srcltx}

\usepackage{amsmath,amssymb,setspace,nicefrac}
\usepackage[active]{srcltx}
\usepackage[pagebackref, colorlinks, linkcolor=red, citecolor=blue, urlcolor=blue, hypertexnames=true]{hyperref}
\usepackage[backrefs]{amsrefs}

\usepackage{amsmath,enumerate,amsxtra,amssymb,latexsym,amscd,amsthm,float,indentfirst}
\usepackage[mathscr]{eucal}
\usepackage[all]{xy}
\usepackage{setspace}
\usepackage{CJK}
\usepackage{tabu, subfigure}
\usepackage{makecell}
\usepackage{mathrsfs}
\usepackage{lineno,hyperref}

\newtheorem{theorem}{Theorem}[section]
\newtheorem*{mthm}{Main Theorem}
\newtheorem{definition}[theorem]{Definition}
\newtheorem{proposition}[theorem]{Proposition}

\newtheorem{lemma}[theorem]{Lemma}
\newtheorem{remark}[theorem]{Remark}

\newcommand{\R}{\mathbb{R}}
\newcommand{\Z}{\mathbb{Z}}
\newcommand{\N}{\mathbb{N}}

\newcommand{\nx}{\mathcal{U}}
\newcommand{\poincare}{Poincar\'e }

\begin{document}

\title[Robustly shadowable chain transtive sets  and hyperbolicity]{Robustly shadowable chain transitive sets \\  and hyperbolicity }
\author{MOHAMMAD REZA BAGHERZAD}
\address{Behin Andishan Aayyar research group, Sarparast st, Taleghani st, Tehran 1616893131, Iran.} 
\email{mbagherzad@gmail.com}
\author{Keonhee Lee}
\address{Keonhee Lee, Department of Mathematics, Chungnam National University, Daejeon 305-764, Korea.}
\email{khlee@cnu.ac.kr}
\date{\today}
\maketitle

\begin{abstract}
We say that a compact invariant set $\Lambda$ of a $C^1$-vector field $X$ on a compact boundaryless Riemannian manifold $M$ is robustly shadowable if it is locally maximal with respect to a neighborhood $U$ of $\Lambda$, and there exists a $C^1$-neigborhood $\nx$ of $X$  such that for any $Y \in \nx$, the continuation $\Lambda_Y$ of $\Lambda$ for $Y$ and $U$ is shadowable for $Y_t$. In this paper, we prove that any chain transitive set of a $C^1$-vector field on $M$ is hyperbolic if and only if it is robustly shadowable.

\end{abstract}


\onehalfspace
\section{Introduction}

\hspace{2mm}
The main goal of the study of differentiable dynamical systems is to understand the structure of the orbits of vector fields (or diffeomorphisms) on a compact boundaryless Riemannian manifold. To descirbe the dynamics on the underlying manifold, it is usual to use the dynamic properties on the tangent bundle such as hyperbolicity and dominated splitting. A fundamental problem in recent years is to study the influence of a robust dynamic property (i.e., property that holds for a given system and all $C^1$-nearby systems) on the behavior of the tangent map on the tangent bundle (e.g., see  \cite{GLT, LS, LTW, LGW,  PT}).

\hspace{2mm}
Recently, several results dealing with the influence of a robust dynamics property of a  $C^1$-vector field were appeared. For instance, Lee and Sakai \cite{LS} proved that a nonsingular vector field $X$ is robustly shadowable (i.e., $X$ and its $C^1$-nearby systems are shadowable) if and only if it satisfies both Axiom $A$ and the strong transversality condition (i.e., it is structurally stable). Afterwards, Pilyugin and Tikhomirov \cite{PT} gave a description of robustly shadowable  oriented vector fields which are structurally stable. In particular, it is proved in \cite{LTW} that any robustly shadowable chain component $C_X(\gamma)$ of $X$ containing a hyperbolic periodic orbit $\gamma$ does not contain a hyperbolic singularity, and it is hyperbolic if $C_X(\gamma)$ has no non-hyperbolic singularity. Here we say that the chain component $C_X(\gamma)$ is {\it robustly shadowable} if there is a $C^1$-neighbohood $\nx$ of $X$ such that for any $Y \in \nx$, the continuation $C_Y(\gamma_Y)$ of $Y$ containing $\gamma_Y$ is shadowable for $Y_t$, where $\gamma_Y$ is the continuation of $\gamma$ with respect to $Y$.
Very recently, Gan {\it et al.} \cite{GLT} showed that the set of all robustly shadowable oriented vector fields is contained in the set of vector fields with $\Omega$-stability. In this direction, the following question is still open: {\it if the chain component $C_X(\gamma)$ of a $C^1$-vector field $X$ on  a compact boundaryless Riemannian manifold $M$ containing a hyperbolic periodic orbit $\gamma$ is robustly shadowable, then is it hyperbolic? }

\hspace{2mm} 
In this paper, we study the dynamics of robustly shadowable chain transitive sets. More precisely, we prove that any chain transitive set of a vector field $X$ is  hyperbolic if and only if it is robustly shadowable. For this, we first show that if a compact invariant set $\Lambda \subset M$ is robustly shadowable then every singularity and periodic orbit in $\Lambda_Y$ are hyperbolic for $Y_t$, where $\Lambda_Y$ is the continuation of  $\Lambda$ with repect to a $C^1$-nearby vector field $Y$. Moreover, we see that any robustly shadowable chain transitive set  $\Lambda$ does not contain a singularity. Finally we show that $\Lambda$ admits a dominated splitting, and it is indeed a hyperbolic splitting.


\hspace{2mm} 
Now we round out the introduction with some notations, definitions and main theorem which we will use throughout the paper.
Let $M$ be a compact boundaryless Riemannian manifold with dimension $n$. Denote by
$\mathcal{X}^1(M)$ the set of all $C^1$ vector fields of $M$ endowed
with the $C^1$ topology. Then every $X\in\mathcal{X}^1(M)$ generates
a $C^1$ flow $X_t:M\times \mathbb{R}\to M$, that is, a family of diffeomorphisms on $M$ such that $X_s\circ X_t=X_{t+s}$ for all $t,s\in\mathbb R$, $X_0=Id$ and $\frac{dX_t}{dt}\vert_{t=0}=X(p)$ for any $p\in M$. Throughout the paper, for
$X,Y,\ldots \in \mathcal{X}^1(M)$, we always denote the generated
flows by $X_t, Y_t, \ldots$, respectively. For $x\in M$, let us denote the orbit $\{X_t(x), t\in\mathbb R\}$ of the flow $X_t$ (or $X$) through $x$ by $orb(x,X_t)$, or $O(x)$ if no confusion is likely.
We say that a point $x\in M$ is a {\it singularity} of $X$ if $X(x)=0$; and an orbit $O(x)$ is {\it closed} (or {\it periodic}) if it is diffeomorphic to a circle $S^1$. Let $d$ be the distance induced from the Riemannian structure on $M$. A
sequence $\{(x_i,t_i):x_i\in M; t_i\geq 1; a<i<b\}$ ($-\infty \leq a < b \leq \infty$) is called a  $\delta$-{\it pseudo orbit} (or a $\delta$-{\it chain}) of $X_t$ if for any $a<i<b-1$,
$$d(X_{t_i}(x_i),x_{i+1})<\delta.$$ 

Roughly speaking, a pseudo orbit is composed by a set of segments of real orbits. We need the restriction $t_i\geq 1$ because without this, for any $\delta>0$, all points $x,y \in M$ can be connected by a $\delta$-pseudo orbit.

\hspace{2mm} 
Let {\it Rep} be the set of all increasing homeomorphisms (called {\it reparametrizations}) $h:\R\rightarrow \R$ such that $h(0)=0.$ We say that a compact invariant
set $\Lambda$ of $X_t$ is {\it shadowable}
if for any $\varepsilon>0$, there
is $\delta>0$ satisfying the following property: given any
$\delta$-pseudo orbit $\{(x_i,t_i):-\infty \leq i \leq \infty\}$ in $\Lambda$,
there exist a point $y\in M$ and $h \in {\it Rep}$ such that for all $t \in \R$ we have 
$$ d(X_{h(t)}(y),x_0*t)<\varepsilon,$$
where $x_0*t=X_{t-S_i}(x_i)$ for any $t \in [S_i,~S_{i+1}]$, and $S_i$ is given by
$$
S_i=
\begin{cases}
\sum_{j=0}^{i-1}t_j&\quad\mbox{for}\quad i>0,\\
0  &\quad\mbox{for}\quad i=0, \\
-\sum_{j=i}^{-1}t_j&\quad\mbox{for}\quad i<0.
\end{cases}
$$

Note that the above concept of pseudo orbit is slightly different from that of pseudo orbit in \cite{LS, PT}. However we point out here that a compact invariant set $\Lambda$  is shadowable for $X_t$ under the above definition if and only if it is shadowable for $X_t$ under the definition in \cite{LS, PT}.
A point $x\in M$ is called {\it chain recurrent} if for any
$\delta>0$, there exists a $\delta$-pseudo orbit $\{(x_i,t_i):0\leq
i<n\}$ with $n>1$ such that $x_0=x$ and
$d(X_{t_{n-1}}(x_{n-1}),x)<\delta$. The set of all chain recurrent
points of $X_t$ is called the {\it chain recurrent set} of $X_t$, and denote it by $CR(X_t)$.  For any $x,y\in M$, we say that $x\sim y$, if
for any $\delta>0$, there are a $\delta$-pseudo orbit
$\{(x_i,t_i):0\leq i<n\}$ with $n>1$ such that $x_0=x$ and
$d(X_{t_{n-1}}(x_{n-1}),y)<\delta$ and a $\delta$-pseudo orbit
$\{(x'_i,t'_i):0\leq i<m\}$ with $m>1$ such that $x'_0=y$ and
$d(X_{t'_{n-1}}(x'_{n-1}),x)<\delta$. It is easy to see that $\sim$
gives an equivalence relation on the set $CR(X_t)$. An equivalence class of
$\sim$ is called a {\it chain component} of $X_t$ (or $X$).
We say that a compact invariant set $\Lambda$ of $X_t$ is {\it chain transitive} if for any $x,y \in \Lambda$ and any $\delta>0$, there is a $\delta$-pseudo orbit $\{(x_i,t_i) \in \Lambda \times\R \mid  t_i\geq1, 0 \leq i <n\}$ with $n>1$ such that $x_0=x$ and $d(X_{t_{n-1}}(x_{n-1}),y)<\delta$.

\hspace{2mm}
A compact invariant set $\Lambda$ of $X_t$ is
called {\it hyperbolic} if there are constants $C>0$ and $\lambda>0$
such that the tangent flow $DX_t:T_\Lambda M\to T_\Lambda M$ leaves
a continuous invariant splitting $T_\Lambda M=E^s\oplus \langle X\rangle\oplus E^u$ satisfying
$$\left\|DX_t|_{E^s(x)}\right\|\leq Ce^{-\lambda t}\quad\mbox{and}\quad \left\|DX_{-t}|_{E^u(x)}\right\|\leq Ce^{-\lambda t}$$
for any $x\in\Lambda$ and $t>0$, where $\langle X\rangle$ denotes the subspace generated by the vector field $X$.
For any hyperbolic closed orbit $\gamma$, the sets $$W^s(\gamma)= \{x\in M: X_t(x) \to \gamma ~{\rm as}~ t \to \infty \} ~ \text{and} $$ 
$$W^u(\gamma)= \{x\in M: X_t(x) \to \gamma ~{\rm as}~ t \to -\infty \}$$
are said to be the {\it stable manifold} and {\it unstable manifold} of $\gamma$, respectively.
 We say that the dimension of the stable manifold $W^s(\gamma)$ of $\gamma$ is the {\it index} of $\gamma$, and denoted by $ind(\gamma)$.

\hspace{2mm}
The {\it homoclinic class} of $X_t$ associated to $\gamma$, denoted by $H_X(\gamma)$, is defined as the closure of the transversal intersection of the stable and unstable manifolds of $\gamma$, that is;
$$H_X(\gamma)=\overline{W^s(\gamma)\pitchfork W^u(\gamma)}. $$
By definition, we easily see that the set is closed and $X_t$-invariant. Let $C_X(\gamma)$ be the chain component of $X_t$ containing a hyperbolic periodic orbit $\gamma$. Then we have $H_X(\gamma)\subset C_X(\gamma)$, but the converse is not true in general.
For two hyperbolic closed orbits $\gamma_1$ and $\gamma_2$ of $X_t$, we
say $\gamma_1$ and $\gamma_2$ are {\it homoclinically related}, denoted by $\gamma_1\sim\gamma_2$, if
$$W^s(\gamma_1)\pitchfork W^u(\gamma_2)\neq \emptyset ~ \text{and} ~
W^s(\gamma_2)\pitchfork W^u(\gamma_1)\neq \emptyset.$$ 
By Birkhoff-Smale's theorem (see \cite{AP}), we know that $$H_X(\gamma)=\overline{\{\gamma':\gamma'\sim\gamma\}}.$$

\hspace{2mm}
A point $x\in M$ is called {\it nonwandering} if for any neighborhood $U$ of $x$, there is $t\geq 1$ such that $X_t(U)\cap U\not= \emptyset$.
The set of all nonwandering points of $X_t$ is called the {\it nonwandering set} of $X_t$, denoted by $\Omega(X_t)$. Let $Sing(X)$ be the set of all singularities of $X$, and let $PO(X_t)$ be the set of all periodic orbits (which are not singularities) of $X_t$. Clearly we have 
$$Sing(X)\cup PO(X_t)\subset \Omega(X_t) \subset CR(X_t).$$
 We say that $X$ satisfies {\it Axiom A} if $PO(X_t)$ is dense in $\Omega(X_t)\setminus Sing(X)$, and $\Omega(X_t)$ is hyperbolic for $X_t$.
A point $y \in M$ is said to be an {\it $\omega$ limit point} of $x$ if there exists a sequence $t_i\rightarrow +\infty$ such that $X_{t_i}(x)\rightarrow y$. Denote the set of all omega limit points of $x$ by $\omega(x)$. We say that a compact invariant set $\Lambda$ of $X_t$ is {\it transitive} if there is $x\in \Lambda$ such that $\omega(x)=\Lambda$.

\hspace{2mm}
Let $\Lambda$ be a compact invariant set of $X_t$. For any $C^1$-close $Y$ to $X$ and a neighbourhood $U$ of $\Lambda$, the set $$\Lambda_Y=\displaystyle\bigcap_{t \in \R}Y_t(U)$$ is called the {\it continuation} of $\Lambda$ for $Y$ and $U$. If there exists a neighbourhood $U$ of $\Lambda$ satisfying $\Lambda=\bigcap_{t \in \R}X_t(U),$ then we say that $\Lambda$ is {\it locally maximal} with respect to $U$, and $U$ is called an {\it isolating block} of $\Lambda$. Let $\gamma$ be a hyperbolic closed orbit of $X_t$. Then we know that there are a $C^1$ neighbourhood $\nx$ of $X$ and a neighbourhood $U$ of $\gamma$ such that for any $Y \in \nx$, there is a unique hyperbolic closed orbit $\gamma_Y$ in $U$ which is equal to  the set $\bigcap_{t \in \R}Y_t(U)$. Note that every $\gamma_Y$ is locally maximal with respect to $U$. The chain component of $Y \in \nx$ containing the continuation $\gamma_Y$  will be denoted by $C_Y(\gamma_Y)$. 

\hspace{2mm}
Now we give the definition of robust shadowability for invariant sets of vector fields.

\begin{definition}
We say that a compact invariant set $\Lambda$ of $X_t$ is {\it robustly shadowable} if it has an isolating block $U$, and there exists a $C^1$-neighborhood $\nx$ of $X$ such that for any $Y \in \nx$, the continuation $\Lambda_Y$ for $Y$ and $U$ is shadowable for $Y_t$. Here $\nx$ is said to be an admissible neighborhood of $X$ with repsect to $\Lambda$.
\end{definition}

\hspace{2mm}
In this paper, we prove the following main theorem.

\begin{mthm}
Let $X \in \mathcal{X}^1(M)$, and let $\Lambda$ be a compact, invariant and  chain transitive set for $X_t$. Then  $\Lambda$ is hyperbolic if and only if it is robustly shadowable.
\end{mthm}

\medskip


\section{Linear \poincare flows and quasi hyperbolic orbit arcs}

\hspace{2mm}
Hereafter we assume that the exponential map $$\exp_p:T_pM(1)\to M$$ is well defined for all
$p\in M$, where $T_pM(r)$ denotes the $r$-ball $\{v\in T_pM: \|v\|\leq r\}$ in $T_pM$. For any regular point $x\in M$ (i.e., $X(x)\neq \bold{0}$), we let
$$N_x=({\rm span\ } X(x))^\perp\subset T_xM,$$ and
$N_x(r)$ the $r$-ball in $N_x$. Let
$\hat{N}_{x,r}=\exp_x(N_x (r))$. Given any regular point $x\in M$ and
$t\in \mathbb{R}$, we can take a constant $r>0$ and a $C^1$ map $\tau:
\hat{N}_{x,r}\to \mathbb{R}$  such that  $\tau(x)=t$ and
$X_{\tau(y)}(y)\in \hat{N}_{X_t(x),1}$ for any $y\in\hat{N}_{x,r}$.
Now we define the {\it Poincar\'e map}
$$f_{x,t}:\hat{N}_{x,r} \to \hat{N}_{X_t(x),1}, ~~ f_{x,t}(y)=X_{\tau(y)}(y)$$
for $y \in \hat{N}_{x,r}.$
Let $M_X=\{x\in M: X(x)\neq \bold{0}\}$. Then it is easy to check that for any
fixed $t$ there exists a continuous map $r_0: M_X\to (0,1)$ such that
for any $x\in M_X$, the Poincar\'e map $f_{x,t}:
\hat{N}_{x,r_0(x)}\to \hat{N}_{X_t(x),1}$ is well defined and the
respective time function $\tau(y)$ satisfies  $2t/3<\tau(y)<4t/3$ for $y\in\hat{N}_{x,r_0(x)}$.

\hspace{2mm}
Let $t_0$ be fixed. At each $x\in M_X$, one can consider a
{\it flow box chart}  $(\hat{U}_{x,t_0,\delta}, F_{x,t_0})$ at $x$ such that
$$\hat{U}_{x,t_0,\delta}=\{tX(x)+y: 0\leq t\leq t_0, y\in N_x(\delta)\} \subset T_xM,$$
where $F_{x,t_0}:\hat{U}_{x,t_0,\delta}\to M$ is defined by
$F_{x,t_0}(tX(x)+y)=X_t(\exp_xy)$. Then it is well known that if
$X_t(x)\neq x$ for any $t\in(0,t_0]$, then there is $\delta>0$ such
that $F_{x,t_0}:\hat{U}_{x,t_0,\delta}\to M$ is an embedding.

For $\varepsilon>0$ and $r>0$, let $\mathcal{N}_{\varepsilon}(\hat{N}_{x,r})$
be the set of all diffeomorphisms $\phi:\hat{N}_{x,r}\to
\hat{N}_{x,r}$ such that 
$$supp(\phi)\subset \hat{N}_{x,r/2} ~ \text{and} ~ d_{C^1}(\phi,id)<\varepsilon. $$
Here $d_{C^1}$ is the usual $C^1$
metric, $id$ denotes the identity map and the $supp(\phi)$ is the
closure of the set of points where it differs from $id$.

\begin{proposition}\label{PR}
Let $X\in\mathcal{X}^1(M)$, and let $\mathcal{U}\subset \mathcal{X}^1(M)$
be a neighborhood of $X$. For any constant $t_0>0$, there
are a constant $\varepsilon>0$ and a $C^1$-neighborhood $\mathcal{V}$ of
$X$ such that for any $Y\in \mathcal{V}$, there exists a continuous
map $r:M_Y\to(0,1)$ satisfying the following property$:$ for any $x\in M_Y$ satisfying $Y_t(x)\neq x$ for $0<t\leq 2t_0$ and any
$\phi\in\mathcal{N}_\varepsilon(\hat{N}_{x,r(x)})$, there is
$Z\in\mathcal{U}$ such that $Y(z)=Z(z)$ for all $z\in M\backslash
F_x(\hat{U}_x)$ and $Z_t(y)=Y_t(\phi(y))$ for any $y\in \hat{N}_{x,r(x)}$ and $2t_0/3<t<4t_0/3$,
where $F_x(\hat{U}_x)$ is the flow box of   $Y$ at $x$.
\end{proposition}

\begin{proof}
See \cite[p. 293--295]{PR}.
\end{proof}

\begin{remark}
In the above proposition, it is easy to see that if $\phi(x)=x$, then $f_{x,t_0}\circ\phi$ is the Poincar\'e map of $Z$, where $f_{x,t_0}:
\hat{N}_{x,r(x)}\to \hat{N}_{X_{t_0}(x),1}$ is the Poincar\'e map of
$Y$.
\end{remark}

\hspace{2mm}
For the study of stability conjecture (see \cite{GW}) posed by Palis and Smale, Liao \cite{L} introduced the
notion of linear Poincar\'e flow for a $C^1$-vector field as follows. Let
$\mathcal{N}=\bigcup_{x\in M_X}N_x$ be the normal bundle based on
$M_X$. Then we can introduce a flow (which is called a {\it linear Poincar\'e flow} for $X$)
$$\Psi_t :\mathcal{N}\to\mathcal{N}, ~\Psi_t|_{N_x}=\pi_{N_x}\circ D_xX_t|_{N_x},$$
where $\pi_{N_x}:T_xM\to N_x$ is the natural projection along the direction of $X(x)$, and $D_xX_t$ is the derivative map of $X_t$. Then we can see that
$$\Psi_t|_{N_x}=D_xf_{x,t}\quad\mbox{and}\quad f_{x,t}\circ\exp_x=\exp_{X_t(x)}\circ \Psi_t.$$

Using Proposition 2.1, we can prove the following lemma
which has the same philosophy with the Franks' Lemma for
diffeomorphisms. One can find another proof for the lemma in \cite{BGV}.

\begin{lemma}
Let $\mathcal{U}$ be a $C^1$ neighborhood of $X\in\mathcal{X}^1(M)$. For any $T>0$, there exists a constant $\eta>0$ such that
for any tubular neighborhood $U$ of an orbit arc
$\gamma=X_{[0,T]}(x)$ of $X_t$ and for any $\eta$-perturbation
$\mathcal{F}$ of the linear Poincar\'e flow $\Psi_T|_{N_x}$, there
exists a vector field $Y\in\mathcal{U}$ such that the linear Poincar\'e flow
$\tilde{\Psi}_T|_{N_x}$ associated to $Y$ coincides with
$\mathcal{F}$, and  $Y$ coincides with $X$ outside $U$ and
along $X_{[-t_1,t_2]}(x)$, where $t_1=\min\{t>0, X_{-t}(x)\in
\partial{U}\}$ and $t_2=\min\{t>0,X_t(x)\in\partial{U}\}$.
\end{lemma}

\hspace{2mm}
We introduce the notions of dominated splitting and hyperbolic splitting for linear Poincar\'e flows as follows.

\begin{definition}
Let $\Lambda$ be an invariant set of $X_t$ which
contains no singularity. We call a $\Psi_t$-invariant splitting
$\mathcal{N}_\Lambda=\Delta^s\oplus \Delta^u$ as an \it{$l$-dominated splitting} (or $\Lambda$ admits an \it{$l$-dominated splitting}) if
$$\left\|\Psi_t|_{\Delta^s(x)}\right\|\cdot\left\|\Psi_{-t}|_{\Delta^u(X_t(x))}\right\|\leq \frac{1}{2}$$
for any $x\in\Lambda$ and any $t\geq l$, where $l>0$ is a constant. Moreover, if $\dim(\Delta^s_x)$ is constant for all $x\in\Lambda$, then we say that the splitting is a homogeneous dominated splitting.
Furthermore, a $\Psi_t$-invariant splitting $\mathcal{N}_\Lambda=\Delta^s\oplus \Delta^u$ is said to be
a {\it hyperbolic splitting} if there exist $C>0$ and $\lambda\in(0,1)$ such
that
$$\left\|\Psi_t|_{\Delta^s(x)}\right\|\leq C\lambda^t\quad\mbox{and}\quad \left\|\Psi_{-t}|_{\Delta^u(x)}\right\|\leq C\lambda^t$$
for any $x\in\Lambda$ and $t>0$.
\end{definition}

\hspace{2mm}
The following proposition which is crucial to prove the hyperbolicity of invariant sets was proved by Doering and Liao \cite{D, L}. For a detailed proof, see Proposition  1.1 in \cite{D}.

\begin{proposition}\label{doering-liao}
Let $\Lambda\subset M$ be a compact invariant set of $X_t$ such that $\Lambda \cap Sing(X)=\emptyset$. Then
$\Lambda$ is hyperbolic for $X_t$ if and only if the linear
Poincar\'e flow $\Psi_t$ restricted on $\Lambda$ has a hyperbolic splitting
$\mathcal{N}_\Lambda=\Delta^s\oplus \Delta^u$.
\end{proposition}

\begin{proposition}
Let $\Lambda$ be a locally maximal set of $X_t$ with an isolating block $U$. Suppose that  $X$ has a  $C^1$-neighbourhood $\nx$ 
such that for any $Y \in \nx$, every periodic orbit and singularity of $Y$  in $U$ are hyperbolic. Then $X$ has a neighbourhood $\tilde{\nx}$, together with two uniform constants $\tilde{\eta}>0$ and $\tilde{T}>1$ such that for any $Y \in \tilde{\nx}$,
\begin{enumerate}[(i)]
\item whenever $x$ is a point on a periodic orbit of $Y_t$ in $U$ and $\tilde{T}\leq t < \infty$, then $$\frac{1}{t}[\text{log }m(\Psi^Y_t\mid_{E^u_x})-\text{log}\Vert \Psi^Y_t\mid_{E^s_x}\Vert]\geq 2\tilde{\eta};$$
\item whenever $P$ is a periodic orbit of $Y_t$ in $U$ with period $T, x \in P$, and whenever an integer $m \geq 1$ and a partition $0 =t_0<t_1<\dots< t_l=mT$ of $[0,mT]$ are given that satisfy $$t_k-t_{k-1}\geq \tilde{T}, ~ k=1,2,...,l,$$ then $$\frac{1}{mT}\displaystyle\sum_{k=0}^{l-1}\text{log}\Vert \Psi^Y_{t_{k+1}-t_k}\mid_{E^s_{X_{t_{k-1}}(x)}}\Vert\leq -\tilde{\eta},$$  and $$\frac{1}{mT}\displaystyle\sum_{k=0}^{l-1}\text{log }m (\Psi^Y_{t_{k+1}-t_k}\mid_{E^u_{X_{t_{k-1}}(x)}})\geq \tilde{\eta}.$$
\end{enumerate}
\label{liao1}
\end{proposition}
\begin{proof}
See Theorem 2.6 in \cite{LGW}.
\end{proof}

\hspace{2mm}
Let $\Lambda \subset M_X$ be a closed invariant set of $X_t$ that has a continuous $\Psi_t$-invariant splitting $\mathcal{N}_{\Lambda}=\Delta^ s\oplus \Delta^u$ with dim $\Delta^s=p$, $1 \leq p\leq dim M-2$. For two real numbers $T>0$ and $\eta>0$, an orbit arc $(x,t)=X_{[0,t]}(x)$ will be called $(\eta,T,p)$-{\it quasi hyperbolic orbit arc} of $X_t$ with respect to the splitting $\Delta^s \oplus \Delta^u$ if $[0,t]$ has a partition $$0=T_0 <T_1<...<T_l=t$$ such that $T\leq T_i-T_{i-1}<2T$, $i=1,2,...,l,$ and the following three conditions are satisfied:
$$\frac{1}{T_k}\displaystyle\sum_{j=1}^k \text{log}\parallel\Psi_{T_j-T_{j-1}}\mid_{\Delta^s(X_{T_{j-1}})(x)} \parallel\leq-\eta, $$
$$\frac{1}{T_l-T_{k-1}}\displaystyle\sum_{j=k}^l \text{log\,}m(\Psi_{T_j-T_{j-1}}\mid_{\Delta^u(X_{T_{j-1}})(x)})\geq \eta ,$$
 $$\text{log}\parallel\Psi_{T_k-T_{k-1}}\mid_{\Delta^s(X_{T_{k-1}})(x)} \parallel- \text{log\,}m(\Psi_{T_k-T_{k-1}}\mid_{\Delta^u(X_{T_{k-1}})(x)})\leq -2\eta,$$ for $k= 1,2,...,l$.

\hspace{2mm}
Liao \cite{L} proved the following shadowing result which says that any quasi hyperbolic orbit arc with close enough end points  can be shadowed by a hyperbolic periodic orbit.

\begin{proposition}
Let $\Lambda$ be a compact invariant set of $X_t$ without singularities. Assume that there exists a continuous invariant splitting $\mathcal{N}_{\Lambda}=\Delta^s\oplus\Delta^u$ with dim $\Delta^s=p$, $1\leq p\leq \text{dim\,}M-2$. Then for any $\eta>0$, $T>0$, and $\varepsilon>0$, there exists $\delta>0$ such that if $(x,\tau)$ is an $(\eta,T,p)$-quasi hyperbolic orbit arc of $X_t$ with respect to the splitting $\Delta^s \oplus \Delta^u$ and $d(X_{\tau}(x),x)<\delta$ then there exists a hyperbolic periodic point $y \in M$ and an orientation preserving homeomorphism $g:[0,\tau] \rightarrow \R$ with $g(0)=0$ such that $d(X_{g(t)}(y),X_t(x))<\varepsilon$ for any $t \in [0,\tau]$ and $X_{g(\tau)(y)}=y.$
\label{quasi}
\end{proposition}

\section{From robust shadowing to dominated splitting}

\hspace{2mm}
In this section, we prove that if a nontrivial chain transitive subset $\Lambda$ of $X_t$  is robustly shadowable, then it admits a dominated splitting. For this,
we first show that any continuition $\Lambda_Y$ of $\Lambda$ does not contain both a non-hyperbolic sigularity and a non-hyperbolic periodic orbit. Next we show that  $\Lambda$ does not contain a singularity. Finally we prove that $\Lambda$ admits a dominated splitting,

\begin{lemma}
Let $\Lambda$ be a chain transitive set of $X_t$. If $\Lambda$ is robustly shadowable, then it is transitive.

\end{lemma}
\begin{proof}
The proof is straightforward.
\end{proof}

\hspace{2mm}
Using the perturbation technique developed by Pugh and Robinson \cite{PR}, Pilyugin and Tikhomirov  \cite{PT} showed that if $M$ is robustly shadowable for  $X_t$ then there is a $C^1$-neighbourhood $\nx$ of $X$ such that for any $Y \in \nx$, every critical element of $Y_t$ is hyperbolic. Here we prove that any continuition $\Lambda_Y$ of a robustly shadowable chain transitive set $\Lambda$ does not contain both a non-hyperbolic sigularity and a non-hyperbolic periodic orbit

\begin{proposition}
Let $\Lambda$ be a robustly shdaowable set of $X_t$. Then there exists a $C^1$-neighbourhood $\nx$ of $X$ such that  for any $Y \in \nx$, every  singularity and periodic orbit of $Y_t$ in $\Lambda_Y$ are hyperbolic for $Y_t$.
\label{hypesing}
\end{proposition}

\begin{proof}

Suppose $\Lambda$ is a robustly shadowable set of $X_t$. Then there exist  a $C^1$-neighborhood $\nx$ of $X$ and a neighborhood $U$ of $\Lambda$ such that for any $Y \in \nx$, the continuation $\Lambda_Y=\displaystyle\cap_{t \in \R}Y_t(U)$ is shadowable for $Y_t$.

\hspace{2mm}
\textbf{Case} $1$: Suppose there is $Y \in \nx$ such that $\Lambda_Y$ contains a non-hyperbolic singularity $\sigma$.
By using the Taylor's theorem, we may assume that in a neighbourhood of $\sigma$ the dynamical system induced by $Y$ is expressed by the following differential equation: 
$$\dot{x}=Ax+ K(x),$$ where $A \in M_{n\times n}(\R)$ and $K:\R^n \rightarrow \R^n$ is a continuous map satisfying $$\displaystyle\lim_{x \rightarrow 0}\frac{K(x)}{\parallel x \parallel^2}=0.$$ Since $\sigma$ is not hyperbolic, there is an eigenvalue $\lambda$ of $A$ with zero real part. First we assume that $\lambda=0$. By changing coordinate, if necessary, we may assume that there is a $n \times n$-matrix $D$ close enough to $A$ such that 

\begin{equation}
D=\left[\begin{array}{ll}0&0\\0&B\end{array}\right],
\label{coor0}
\end{equation}
where $B$ is a $(n-1) \times (n-1)$-matrix with real entries.
We represent the coordinates of a point $x$ in a neighbourhood of $\sigma$ by $x=(y,z)$ with respect to $D$.
Let $\varepsilon>0$, and choose a real valued $C^{\infty}$ bump function $\beta:\R \rightarrow \R$ that satisfies the following conditions:
$$
\left\{\begin{array}{l}
\beta(x) \subset [0 ,1]\hspace{1.4 cm}\text{for }x \in \R,\\
\beta(x)=0  \hspace{2 cm} \text{for }\mid x\mid \geq \varepsilon,\\
\beta(x) =1\hspace{2 cm} \text{for }\mid x\mid \leq \frac{\varepsilon}{4},\\
0\leq\beta'(x)<\frac{2}{\varepsilon} \hspace{1.1 cm}\text{for } x \in \R.
\end{array}\right.
$$
Define $\rho:\R^n \rightarrow \R$ by $\rho(x)=\beta(\parallel x
\parallel)$. By taking $\varepsilon$ small enough, one can see that the vector field $Z$ obtained from the following differential equation $$\dot{x}=Dx+(1-\rho(x))K(x)$$ is $C^1$-close to $Y$. Moreover, we have $B_{\frac{\varepsilon}{4}}(\sigma)\subset U$. Consequently we see that  $Z \in \nx$, $\sigma\in Sing(Z)\cap\Lambda_Z$ and $\Lambda_Z$ is shadowable for $Z$. Since $\rho (x)=1$  for $\parallel x \parallel < \frac{\varepsilon}{4}$ , in the $\frac{\varepsilon}{4}$ neighbourhood of $\sigma$, the differential equation associated to $Z$ is given by $$\left\{\begin{array}{l}\dot{y}=0 \\ \dot{z}=Bz\hspace{.3cm}.\end{array}\right.$$  By considering coordinates represented in (\ref{coor0}), for any $x=(y,z) \in B_{\frac{\varepsilon}{4}}(\sigma),$ we have 
 $$Z_t(x)=Z_t(y,z)=(y,exp(Bt)z).$$
This implies that if $\mid y \mid \leq \frac{\varepsilon}{4}$ then $(y,0) \in Sing(Z) \cap U$, and so $\{(y,0):\,\mid y\mid < \frac{\varepsilon}{4}\} \subset \Lambda_Z.$ 
Let $\delta>0$ be a corresponding constant from the definition of shadowing of $\Lambda_Y$ for $\frac{\varepsilon}{8}$.
Choose $\alpha_0=0<\alpha_1<...<\alpha_n = \frac{\varepsilon}{2}$ such that \ $\mid \alpha_i - \alpha_{i-1} \mid <\delta$ for $i=1,...n$. Let 
$$x_i=(y_i,z_i) ~{\text{and}}~ t_i=1 ~ {\text{for}}~ i=1,...,n.$$
Clearly $\{(x_i,t_i)\mid i=0,\ldots,n\}$ is a finite $\delta$-pseudo orbit of $Z_t$ in $\Lambda_Z$. Since $x_0$ and $x_n$ are singularities we can put 
$$x_i=x_0, ~ t_i=1 ~ \text{for} ~ i \leq 0; ~ {\text{and}} ~ x_i=x_n, ~ t_i=1 ~ \text{for} ~ i>n.$$

Then $\{(x_i,t_i)\mid i\in \Z \}$ is a $\delta$-pseudo orbit of $Z_t$ in $\Lambda_Z$. Since $\Lambda_Z$ is shadowable, there are $(y,z) \in M $ and a reparametrization $h$ such that 
$$
d(X_{h(t)}(y,z),x_0*t)<\frac{\varepsilon}{8}
$$
 for all $t \in \R$. This implies $O(y) \subset B_{\frac{\varepsilon}{4}}(0)$.
Since the intersections of  planes formulated by $\{(y,z) \mid y=c\}$ with $B_{\frac{\varepsilon}{4}}(0)$ are invariant ($c$ is a constant), there is $c_0 \in (- \frac{\varepsilon}{4}, - \frac{\varepsilon}{4})$ such that $O(y) \subset \{(y,z) \mid y=c_0\}$. Without loss of generality, we may assume $c_0=0$. Then we get a contradiction since  $d(X_t(y,z), x_1)\geq\frac{\varepsilon}{4}$ for all $t \in \R$.

\hspace{2mm}
Suppose that $\lambda=ib$ for some nonzero $b\in \R$. By the same techniques as above, we can construct a vector field $Z$ which is $C^1$-close to $Y$ and in a neighbourhood of $\sigma$, the differential equation associated to $Z$ is given by
\begin{equation}
\dot{x}=Ax=\left[\begin{array}{ll}C &0\\0&B\end{array}\right]\left[\begin{array}{l}y\\z\end{array}\right],
\label{coor1}
\end{equation}
where $C=\left[\begin{array}{ll}\text{cos }(b) &\text{sin }(b)\\\text{-sin }(b)&\text{cos }(b)\end{array}\right]$.  By  considering the coordinates obtained from (\ref{coor1}) in the $\frac{\varepsilon}{4}$ neighbourhood of $\sigma$, we can see that  every point $x=(y_1,y_2,0)$ is periodic. Since the intersections of cylinders formulated by $\{(y_1,y_2,z)\mid y_1^2+y_2^2=c, c \in \R \}$  and $B_{\frac{\varepsilon}{4}}(\sigma)$ are invariant, we can derive a contradiction by using the same techniques as above.

\hspace{2mm}
   \textbf{Case} $2$: Suppose there is $Y \in \nx$ such that $\Lambda_Y$ contains a non-hyperbolic periodic orbit $\gamma$. Let $p \in \gamma$, and denote the period of $\gamma$ by $\pi(p)$. 
Then the linear \poincare map $\Psi_{\pi(p)}:N_p \rightarrow N_{p}$ has an eigenvalue of modulus $1$. Hence we can find a linear map $P:N_p\rightarrow N_p $ arbitrarily close to $\Psi_{\pi(p)}$ that has an eigenvalue $\lambda$ of modulus $1$, the multiplicity of $\lambda$ is $1$, and $\lambda$ is a root of unity (i.e., $\lambda^n=1$ for some $n \in \N$). 
Using Lemma 2.3, we may assume that $\Psi_{\pi(p)}=P$.
By changing the coordinates in $N_P$, if necessary, we may assume that
\begin{equation}
\Psi_{\pi(p)}=\left[\begin{array}{ll}C &0\\0&B\end{array}\right]
\label{coorpoin3}
\end{equation}
and  $ C\,w =\lambda w$  for some $(w,0) \in N_p$, where $C$ is a $1 \times 1$ (or $2 \times 2$)-matrix. Choose $r>0$ such that $\hat{N}_r \subset U$ and the \poincare map $f_{p,\pi(p)}: \hat{N}_{x,r}\rightarrow \hat{N}_{p,1}$ is well defined. Since $f_{p,\pi(p)}$ is a $C^1$ map, using  the same techniques as in Case  $1$, we can find a map $$g_{p,\pi(p)}: \hat{N}_{x,r}\rightarrow \hat{N}_{p,1}$$ 
which is arbitrarily $C^1$-close to $f_{p,\pi(p)}$ and  $exp_p^{-1}\circ g \circ exp_p\mid_{N_{x,\frac{r}{2}}}=\Psi_{\pi(p)}\mid_{N_{x,\frac{r}{2}}}.$ 
By Proposition 2.1, we may assume that $f_{p,\pi(p)}=g$.

\hspace{2mm}
By the tubular flow theorem for closed orbits in Section 2.5.2 in \cite{AP}, we can find constants $s,\delta_0, l>0$ such that if $x \in \hat{N}_p\cap B_{s}(p)$,  $y \in M$ and $\varepsilon \in (0,\delta_0)$ then $d(x,y)<\varepsilon$ implies $y=Y_{t'}(y'),$ for some $y' \in \hat{N}_p$ and $\mid t' \mid , d(y',x)<l\varepsilon.$
Let $\delta>0$ be a corresponding constant for $\varepsilon<\text{min}\{\delta_0 ,\frac{s}{4l}\}$ obtained from the shadowing property of $\Lambda_Y$. Let $v$ be a scalar multiplication of $w$ which obtained in equation (\ref{coorpoin3}) satisfying $\parallel v \parallel=s$.
To make a $\delta$-pseudo orbit, fix $N>0$ and define 
$$
x_i=\left\{\begin{array}{l}
p\hspace{4cm} i\leq 0,\\
exp_p(\frac{i}{N}C^iv,0) \hspace{1.6 cm} 0\leq i \leq N-1,\\
exp_p(C^Nv,0) \hspace{1.8cm} i\geq N
,\end{array}\right.$$ 
and $t_i=\tau(x_i),$ where $\tau$ is the first return map. Then we get $$d(X_{t_i}(x_i),x_{i+1})= \parallel \frac{i}{N}C^{i+1}v-\frac{i+1}{N}C^{i+1}v\parallel=\parallel \frac{\lambda^{i+1}}{N}v\parallel=\parallel \frac{1}{N}v\parallel<\delta,$$ for sufficient large $N$. Since $C^nv=\lambda^nv=v$, we see that each $\{x_i\}$ is periodic and $O(x_i) \subset U$ for all $i \in \Z$. Consequently, we get $x_i \in \Lambda_Y$ for all $i \in \Z$. Since $\Lambda_Y$ satisfies the shadowing property, there are $x \in M$ and $h \in \text{Rep}$ such that 
$$
O(x) \subset B_s(\gamma)\text{ and } d(X_{h(t)}(x),x_0*t)<\varepsilon
$$
 for all $t \in \R$.  Hence there are $t_1,t_2 \in \R$ such that $$d(Y_{t_1}(x),p)<\varepsilon \text{ and } d(Y_{t_2}(x),x_N)<\varepsilon.$$  By the above fact, we can choose $t'_1$ and $t'_2$ in $\R$ such that
\begin{equation}
d(Y_{t'_1}(x),p)<l\varepsilon<\frac{s}{4}, ~  d(Y_{t'_2}(x),x_N)<l\varepsilon<\frac{s}{4}, \text{and} Y_{t'_1}(x), Y_{t'_2}(x) \in \hat{N}_{p}.
\label{moollla}
\end{equation}
Suppose that $$Y_{t'_1}(x)= \text{exp}_p(v_1,w_1)\text{ and } Y_{t'_2}(x) =\text{exp}_p(v_2,w_2).$$ Then (\ref{moollla}) implies that 
\begin{equation}\parallel (v_1,w_1)\parallel < \frac{s}{4} \text{  and  } \parallel (v_2,w_2)-(C^Nv,0)\parallel <\frac{s}{4}. \label{Aha}\end{equation}
Moreover, we see that  $(v_1,w_1)$ and $(v_2,w_2)$ belongs to the same orbit of $\Psi_{\pi(p)}$. Hence, without loss of generality, we may assume that there is $j \in \N$ such that $v_1=C^{j}v_2$. Consequently, we get 
$$\parallel v_1\parallel=\parallel C^j v_2\parallel=\parallel v_2\parallel.$$
But  (\ref{Aha}) implies that $$\parallel v_1\parallel < \frac{s}{4} \text{  and  } \parallel v_2-C^Nv\parallel <\frac{s}{4}. $$  On the other hand, we have  $\parallel C^nv\parallel=\parallel v\parallel=s$, and so the  contradiction completes the proof of our proposition.  
\end{proof}

\hspace{2mm}
Recently, Gan {\it et al.} \cite{GLT} showed that if $M$ is  robustly shadowable for $X \in \mathcal{X}^1(M)$, then there is no singularity $\sigma \in Sing(X)$ exhibiting homoclinic connection. Here the {\it homoclinic connection} is the closure of a orbit of a regular point which is contained in both the stable and the unstable manifolds of $\sigma$.

\begin{proposition}
Let $\Lambda$ be a nontrivial chain transitive set of $X_t$. If $\Lambda$ is robustly shadowable then it  does not contain a singularity of $X$.
\end{proposition}

\begin{proof}
Let $U$ be an isolating block of $\Lambda$, and suppose $U$ contains a singularity $\sigma$. By Proposition 3.2, it must be hyperbolic.

\hspace{2mm}
First we show that there is $z \in W^s(\sigma)\cap W^u(\sigma)$ such that $$\Gamma :=\{\sigma\}\cup O(z) \subset \Lambda.$$
Choose $x \in \Lambda\setminus \{\sigma\}$, and let $\eta>0$ be a constant to ensure that the local stable manifold $W_\eta^s(\sigma)$ and the local  unstable manifold $W_\eta^u(\sigma)$ of $\sigma$ are embedded submanifolds of $M$. Take $\delta_0>0$ satisfying $\displaystyle\bigcup_{y \in \Lambda}B_{\delta_0}(y) \subset U$. Let $\delta>0$ be a corresponding constant for $\varepsilon=min\{\frac{\eta}{2}, \frac{\delta_0}{2}, \frac{d(\sigma,x)}{2}\}$ obtained from the shadowability of $\Lambda$. Since $\Lambda$ is  transitive, there are two finite $\delta$-pseudo orbits in $\Lambda$
\begin{center}$\{(x'_i,t'_i)\mid t'_i\geq1,i=1,\ldots,n\}$ and  $\{(x''_i,t''_i)\mid t''_i\geq 1,i=1,\ldots,m\}$ \end{center} such that $x'_0=x''_m=\sigma$, and $x'_n=x''_0=x$. Define an infinite $\delta$-pseudo orbit in $\Lambda$ as follows: $$(x_i,t_i)=\left\{\begin{array}{l}
(\sigma,1)\hspace{2cm} i< 0,\\
(x'_i,t'_i) \hspace{1.8cm} 0\leq i< n,\\
(x''_{i-n},t''_{i-n}) \hspace{1cm} n\leq i < n+m, \\
(\sigma,1)\hspace{2cm} i\geq n+m.
\end{array}\right.$$ 
Then there are $z \in M$ and $h \in \text{Rep}$ such that $$d(X_{h(t)}(z),x_0*t)<\varepsilon$$ for all $t \in \R$. This implies that there is $T>0$ such that 
$$d(X_{t}(z),\sigma)<\eta  ~ {\text{for all}}~ t>T ~ {\text{and}}~ t<-T.$$
By our construction, we see that $z \in W^s(\sigma)\cap W^u(\sigma).$  Hence we have $$\displaystyle\sup_{t \in \R, y \in \Lambda}(d(X_t(z),y))<\varepsilon<\frac{\delta_0}{2}.$$ This implies that $O(z) \subset U$ and $z \in \Lambda$.

\hspace{2mm}
Second we show that there is $x \in W^s(\sigma)\cap W^u(\sigma)$ such that $$\Gamma' :=\{\sigma\}\cup O(x) \subset \Lambda ~\text{and}~ x \notin O(z).$$ Let $\varepsilon>0$ be such that $\bigcup_{x\in\Gamma}B_{\varepsilon}(x) \subset U,$ and let $\delta$ be a corresponding constant for $\varepsilon$ obtained from the shadowing property of $\Lambda$. Since $z \in W^s(\sigma)\cap W^u(\sigma)$, there is $m \in \N$ such that 
$$d(X_n(z),\sigma)<\frac{\delta}{2} ~\text{and}~ d(X_{-n}(z),\sigma)<\frac{\delta}{2}$$ for all $n\geq m$. Consider a $\delta$-pseudo orbit in $\Lambda$  $$(x_i,t_i)=\left\{\begin{array}{l}
(X_i(z),1)\hspace{1.6cm}\text{for}\hspace{1cm} i\leq m ,\\
(X_{i-2m}(z),1)\hspace{1cm}\text{for}\hspace{1cm}i>m.
\end{array}\right.$$ 
Then there are $x \in M$ and $h \in \text{Rep}$ such that $d(X_{h(t)}(x),x_0*t)<\varepsilon.$ We also easily check that $$
x \in W^s(\sigma)\cap W^u(\sigma), ~ O(x) \subset U ~ \rm{and} ~ x \not\in O(z).
$$ 
This implies that $dimE^s=dimW^s(\sigma) = k \geq2$. By applying Lemma 3.5 in \cite{GLT}, we can assume that there is a dominated splitting $E^s=E^c \oplus E^{ss}$ such that $dimE^c=1$. We also perturb $\Gamma$ and $\Gamma'$ to make sure that $$(\Gamma\cup \Gamma')\cap W^{ss}(\sigma)=\{\sigma\},$$ 
where $W^{ss}(\sigma)$ be the strong stable submanifold of $M$ tangent to $E^{ss}$.
Furthermore we may perturb that in a neighbourhood $V$ of $\sigma$, the dynamic induced by $X$ is expressed by the following differential equation
\begin{equation}\left[\begin{array}{l}\dot{x}^c\\\dot{x}^{ss}\\\dot{x}^u\end{array}\right]=A\left[\begin{array}{l}x^c\\x^{ss}\\x^u\end{array}\right]=\left[\begin{array}{lll}B_1 & 0 & 0 \\ 0&B_2 & 0\\ 0&0&C\end{array}\right]\left[\begin{array}{l}x^c\\x^{ss}\\x^u\end{array}\right],
\label{coor3}
\end{equation}
where $B_1$, $B_2$ and $C$ are preserving the splitting  $E^c\oplus E^{ss} \oplus E^u$. Here the eigenvalues of  $B_2$ and $C$ have  negative and positive real part, repectively, and the spectrum of $B_1=\{\lambda_1\}$. For more details on these perturbations, see  \cite{GLT}.
Since the dynamic  on $V$ is induced by the differential equation $\dot{x}=Ax$, we can express every point $y$ in $V$ by $y=(y^c,y^{ss},y^u)$  based on the coordinates obtained from $E^c\oplus E^{ss}\oplus E^u$. Then we get  
\begin{equation}
X_{t}(y)= (X_{t}(y^c),X_t(y^{ss}), X_{t}(y^u))=(e^{B_1 t}y^c,e^{B_2t}y^{ss}, e^{Ct}y^u).
\label{coor5}
\end{equation}

\hspace{2mm}
Next we are going to get some useful properties for $\Gamma \cup \Gamma'$ that helps us to complete the proof. 
 Choose $x \in \Gamma'$ and $z_1,z_2 \in \Gamma$ satisfying
 $$x,z_1 \in W^s(\sigma), ~ z_2\in W^u(\sigma), ~O^+(x) \cup O^+(z_1)\subset V, ~\text{and}~ O^-(z_2) \subset V. $$

Fix $r>0$ and let $y \in \hat{N}_{x,r} $. Assume that there exists $t>0$ such that $$X_{t}(y) \in \hat{N}_{z_2,r} \text{ and } X_{[0,t]}\subset V.$$ For any $y \in \hat{N}_{x,r} ,$ denote by $\tau(y)$  the minimum of $t$ with the above property (if such a $t$ exists).  Define a map $P_r$ by
$$P_r : \text{Dom}(P_r) \subset \hat{N}_{x,r} \rightarrow \hat{N}_{z_2,r},  ~ ~ P_r(y)=X_{\tau(y)}(y).$$
We show that there is $r_0>0$ such that  Dom$(P_r) \neq \emptyset$ for any $r \in (0,r_0]$. 
Fix $r_0>0 $  such that  
$$\displaystyle\bigcup_{t\geq 0}\{B_{r_0}(X_t(x))\cup B_{r_0}(X_{-t}(z_2))\} \subset V.$$  
Let $r \in (0 ,r_0]$, and take $r_1 \in (0 ,r]$ such that  $d(X_t(y),x)<r_1$ for $y \in M$ and $t>0$. Then there is $t'\in [0,t]$  such that 
$$X_{t'}(y) \in \hat{N}_{x,r} ~ \text{and} ~ X_{[t',t]}(y)\subset B_r(x).$$
If $d(X_t(y),z_2)<r_1$, then  there is $t' \in [t , \infty)$ such that 
$$X_{t'}(y) \in \hat{N}_{z_2,r} ~\text{and} ~ X_{[t,t']}(y)\subset B_r(z_2).$$

Let $\delta>0$ be a corresponding constant for $\frac{r_1}{2}$ obtained from the shadowing property of $\Lambda$. Let $m \in \N$ be such that 
$$d(X_{m}(x),\sigma)<\frac{\delta}{2} ~\text{and} ~ d(X_{-m}(z_2),\sigma)<\frac{\delta}{2}.$$
Consider the following $\delta$-pseudo orbit 
\begin{equation}
(x_i,t_i)=\left\{\begin{array}{l}(x,1) \hspace{2 cm} i\leq m,\\
(X_{-m}(z_2),1) \hspace{.7 cm} i\geq m+1.
\end{array}\right.
\label{pseudo1}
\end{equation} 
Then there are $y \in M$ and $h \in Rep$ such that $$d(X_{h(t)}(y),x_0*t)<\frac{r_1}{2}.$$ 
This implies that there are  $0\leq t_1 <t_2<t_3 < \infty$ such that 
$$d(X_{[h(t_1),h(t_2))}(y),X_{[0 ,m]}(x))<\frac{r_1}{2} ~\text{and} ~d(X_{[h(t_2),h(t_3)]}(y),X_{[-m ,0]}(z_2))<\frac{r_1}{2}.
$$
Hence we have $X_{[h(t_1) , h(t_3)]}(y) \subset V $. Let $t'_1,t'_2$ be constants corresponding to  $h(t_1), h(t_3)$, respectively, obtained from the same way we get $r_1$. Then we get
\begin{center}$X_{t'_1}(y) \in \hat{N}_{x,r}$, $X_{t'_2}(y) \in \hat{N}_{z_2,r}$, and $X_{[t'_1 , t'_2]}(y) \subset exp_{\sigma}(T_{\sigma}M(1)) $.\end{center} 
Consequently, we have 
$y \in ~\text{Dom}(P_r)$ and so Dom$(P_r) \neq \emptyset$.

\hspace{2mm}
Consider the following set 
$$L =\{(y^c,y^{ss},y^u) \mid y^{ss}=0\}\subset \hat{N}_{z_2,r}.$$  
We will show that for any $\varepsilon>0$ there is $r>0$ satisfying
\begin{equation}
P_r(\hat{N}_{x,r}) \subset \mathcal{C}_{\varepsilon} :=\{(u,w)\in N_{z_2} \mid u\in L , w \in L^{\perp}, \parallel w \parallel\leq \varepsilon \parallel u \parallel\}.
\label{cone1}
\end{equation}
 Let $y\in ~\text{Dom}(P_r) \cap \hat{N}_{x,r}$. Since $P_r(y) \in \hat{N}_{z_2,r}$, we have 
$$0< \Vert z_2^u \Vert -r  \leq \Vert P_r(y)^u\Vert$$
for sufficiently samll $r>0$.  
Using (\ref{coor5}), we get $$\parallel P_r(y)^u \parallel \leq e^{C\tau(y)}\parallel y^u \parallel. $$ Hence $\tau(y) \rightarrow +\infty$ as $\parallel y^u \parallel \rightarrow 0$.
 On the other hand, we have
 \begin{equation}\frac{\parallel P_r(y)^{ss}\parallel }{\parallel P_r(y)^c \parallel}\leq e^{\parallel B_2 \parallel\tau(y)}\frac{\parallel y^{ss}\parallel}{\parallel  e^{B_1\tau(y)}y^c \parallel }.
 \label{inequ1}
 \end{equation}  Since $x^c\neq 0$, we get $y^c \not\rightarrow 0$ as $y \rightarrow x$. In addition, because $E^c\oplus E^{ss}$ is a dominated splitting, the  right side of (\ref{inequ1}) tends zero as $\tau(y) \rightarrow +\infty$, and (\ref{cone1}) is proved.

\hspace{2mm}
Next we perturb $X$ so that if $z_1=X_{t'}(z_2)$ then $\Psi_t(L) \cap \Delta^s=\emptyset  $, where $\Delta^s=N_{z_1}\cap T_{z_1}W^s(\sigma)$. If $\Psi_{t'}(L) \not\subset \Delta^s$ we have nothing to prove. Otherwise, let $u \in N_{z_1}$ be such that $u \not\in \Delta^s$. Fix $\alpha>0$, and denote $u_{\alpha}=\alpha u+(1-\alpha) v$, where $\Psi_{t'}(L)=\rm{Span}\{v\}$. Then there is a linear map $H_{\alpha}:N_{z_1} \rightarrow N_{z_1}$ such that 

$$
H_{\alpha}(v) = u_{\alpha} ~{\text{and}}~ \parallel H_{\alpha} \parallel \rightarrow 1 ~ {\text{as}}~  \alpha    
 \rightarrow 0.
$$

Define  a map 
$$\Psi':N_{z_2} \rightarrow N_{z_1},~ \Psi'(v)=H_{\alpha}\circ\Psi_{t'}(v).$$ Choose $\alpha>0$ so small that we  can use Lemma 2.3, and replace  $\Psi_{t'}$ with $\Psi'$. Then we get  $$\Psi'(L) \cap \Delta^s= \text{Span}\{u_{\alpha}\} \cap \Delta^s= \{\bold{0}\}.$$ Since the \poincare map $f_{z_1,t}:\hat{N}_{z_2,r} \rightarrow \hat{N}_{z_1,r}$ is continuous, there is $\varepsilon>0$ such that   
$$f_{z_1,t}(\mathcal{C}_{\varepsilon})\cap W^s(\sigma)\cap\hat{N}_{z_1,r}=\{z_1\},$$
where $\mathcal{C}_{\varepsilon} $ is defined in (\ref{cone1}). Let $r>0$  be such that $r$ satisfies (\ref{cone1}) for $\varepsilon$, and let $\delta>0$ be a corresponding constant for $\varepsilon'=min\{r,\varepsilon , \eta \}$ obtained from the shadowing property of $\Lambda$.
 Consider the $\delta$-pseudo orbit (\ref{pseudo1}) we constructed in the above. Then there are $y\in M$ and $h \in$ Rep such that $$d(X_{h(t)}
(y),x_0*t)<\varepsilon'.$$
This implies that there are constants $0<t_1<t_2<t_3<t_4$ satisfying 
$$\begin{array}{l}d(X_{h(t_1)}(y),x)<\varepsilon',
d(X_{[h(t_1),h(t_2))}(y),X_{[0,m)}(x))<\varepsilon',\\
d(X_{[h(t_2),h(t_3))}(y),X_{[-m,0]}(z_2))<\varepsilon',
d(X_{[h(t_3),h(t_4)]}(y),X_{[0,t']}(z_2))<\varepsilon',\text{ and}\\
d(X_{[h(t_4),\infty)}(y),X_{[0,\infty)}(z_1))<\varepsilon'.\\
\end{array}$$ 
Without loss of generality, we may assume that \begin{center}$X_{h(t_1)}(y) \in \hat{N}_{x,r}$, $X_{t_3}(y) \in \hat{N}_{z_2,r}$, and $X_{h(t_4)}(y) \in \hat{N}{z_1,r}$.\end{center} This means that $X_{t_3}(y)=P_r(X_{h(t_1)}(y))$, and so we have $X_{h(t_3)}(y) \in \mathcal{C_{\varepsilon}}$. Consequently, we get 
$$X_{h(t_4)} \not\in W^s(\sigma)\cap \hat{N}_{z_1,r}.$$ This is a contradiction to the fact that  $$d(X_{[h(t_4),\infty)}(y),X_{[0,\infty)}(z_1))<\eta,$$
and so  completes the proof.
 \end{proof}

\begin{proposition}\label{pr34}
Let $\Lambda$ be a  chain transitive set. If $\Lambda$ is robustly shadowable, then it admits a homogeneous dominated splitting for $\Psi_t$.
\end{proposition}

\begin{proof}
If $\Lambda$ is a periodic orbit, then  it admits a dominated splitting for $\Psi_t$ by Proposition \ref{hypesing}. Hence we suppose $\Lambda$ is not a periodic orbit, and take a point $x \in \Lambda$ be such that $\omega(x)=\Lambda$. By applying the Pugh's closing lemma (see \cite{PR}),  we can select a sequence $\{Y^n\}_{n \in \N}\subset \nx$  converging to $X$ such that each $Y^n$ has a periodic point $p_n$ converging to $x$; and  for each $t>0$, the sequence 
$\phi_n : [0,t] \rightarrow M$ given by $\phi_n(s) = Y^n_s(p_n)$ converges to $\phi : [0,t] \rightarrow M,~ \phi(s) =  X_s(x) $. Note that here  $O(p_n)$ is hyperbolic for $Y^n_t$ for every $n$.
Moreover we can see that the period of $p_n$ tends to $\infty$ as $n\rightarrow \infty$. By applying Proposition \ref{liao1}, we can take $l>0$ such that the linear \poincare flow of $Y_n$ over $O(p_n)$ admits an $l$-dominated splitting. By taking a subsequence, if necessary,  we may assume that there is $k \in \N$ such that $ind(p_n)=k$ for all $n \in \N$.

\hspace{2mm}
Let $\{x_k \}$ be a sequence in $\Lambda$ converging to $x$, and let $E(x_k)$ be an $m$-dimensional subspace of $T_{x_k}M$. We say that $E(x_k)$ converges to $E(x)$   if, for each k, there is a basis $\{e^1_{k},\ldots ,e^m_k\}$ of $E(x_k)$ and a basis $\{e^1,\ldots,e^m\}$ of $E(x)$ such that $e^i_k\rightarrow e^i$ for each $i = 1, \cdots, m$. 

\hspace{2mm}
Put
$$\lim_{n \rightarrow \infty}E^s_n(p_n) = \Delta^s(x) ~ \text{and} ~  \lim_{n \rightarrow \infty}E^u_n(p_n)=
\Delta^u(x).$$
For each $t>0$, we denote by
$$
\lim_{n \rightarrow \infty}E^s_n(Y^n_t(p_n)) = \Delta^s_n(X_t(x)) ~  \text{and} ~ \lim_{n \rightarrow \infty}E^u_n(Y^n_t(p_n)) = \Delta^u_n(X_t(x)),
$$
where $T_{Y^n_t(p_n)}M=E^s_n(Y^n_t(p_n))  \oplus E^u_n(Y^n_t(p_n))$.
Then we have
$$
\Delta^{s}(X_t(x)) = \lim_{n \rightarrow \infty}\Delta^{s}_n(Y^n_t(p_n)) = \lim_{n \rightarrow \infty}\Psi^n_t(\Delta^{s}_n(p_n)) = \Psi_t(\Delta^{s}(X_t(x))), ~ \rm{and}  $$
$$\Delta^{u}(X_t(x)) = \lim_{n \rightarrow \infty}\Delta^{u}_n(Y^n_t(p_n)) = \lim_{n \rightarrow \infty}\Psi^n_t(\Delta^{u}_n(p_n)) = \Psi_t(\Delta^{u}(X_t(x))), 
$$ 
where $\Psi^n_t$ is the linear \poincare flow for $Y^n$.
This means that the splitting  $\Delta^s(x) \oplus \Delta^u(x)$ is $\Psi_t$ invariant, and we have $\mathcal N_x=\Delta^s(x)\oplus \Delta^u(x)$. If $t$ is sufficiently large, then we can see that
$$\parallel \Psi_t\mid_{\Delta^s(x)}\parallel\cdot \parallel \Psi_{-t}\mid_{\Delta^u(X_{t}(x))}\parallel = \lim_{n \rightarrow \infty}\parallel \Psi^n_t\mid_{\Delta_n^s(x)}\parallel\cdot \parallel \Psi^n_{-t}\mid_{\Delta_n^u(X_{t}(x))}\parallel\leq \frac{1}{2}.
$$
This means that the orbit $O(x)$ admits a dominated splitting for $\Psi_t$, and so $\Lambda = \overline{O(x)}$ also has a dominated splitting for $\Psi_t$, 
\end{proof}


\section{From dominated splitting to hyperbolicity}

\begin{lemma}
If a chain transitive set $\Lambda$ of $X_t$ is robustly shadowable, then it admits a hyperbolic periodic orbit.
\end{lemma}

\begin{proof}
Let $\Delta^s\oplus \Delta^u$ be the $l$-dominated splitting of $(T_{\Lambda}M, \Psi_t|_{N_{\Lambda}})$ obtained in Proposition \ref{pr34}. By using lemma $3.4$ in \cite{GW} and Theorem \ref{liao1} we may assume that  ${dim\,}(\Delta^s)\leq {dim\,}M-2$. Denote by $$\alpha=min\{\parallel \Psi_t \mid_{N_z}\parallel \mid \, z\in \Lambda, t \in [-3,3] \}.$$
For any $\varepsilon>0$, choose $\varepsilon' \in (0,\frac{\alpha}{2})$, $\delta'>0$, and $Y \in \nx$ having a periodic point $p$ such that
\begin{equation}
\left\{\begin{array}{l}
\text{log}(s+\varepsilon') \leq \text{log}(s) + \varepsilon, \hspace{5mm} \forall s \in [\frac{\alpha}{2},\infty),\\ 
\text{log}(\frac{1}{s-\varepsilon'}) \geq \text{log}(\frac{1}{s}) - \varepsilon, \hspace{5mm} \forall s \in [\frac{\alpha}{2},\infty),\\
\mid\,\parallel \Psi_t\mid_{\Delta^{s(u)}(z)} \parallel-\parallel \Psi'_t\mid_{\Delta_Y^{s(u)}(y)} \parallel\,\mid <\varepsilon', \hspace{5mm} \forall  t \in [-3,3],~ d(z,y)<\delta', ~ z \in \Lambda , ~y \in O(p),\\
d_{H}(O(p),\Lambda)<\delta',\\
\end{array}\right.
\label{equ3}
\end{equation}
where $\Psi$ and $\Psi'$ are linear \poincare flows of $X$ and $Y$, respectively. Since $p$ is a hyperbolic periodic point of $Y_t$, there are $C>0$ and $\lambda \in (0,1)$ such that $$\parallel \Psi'_t\mid_{\Delta_Y^{s}(y)} \parallel \leq C\lambda^t \text{ and }\parallel\Psi'_{-t}\mid_{\Delta_Y^{u}(y)}\parallel\leq C\lambda^t $$  for all $t\geq 0$ and $y \in O(p)$. Denote by $C'={\rm{max}}\{C ,C^{-1}\}$, and
let $\delta$ be a constant as in Proposition \ref{quasi} for the triple $(\varepsilon,T,\eta)=(\varepsilon,1, -(log(c')+\varepsilon)) .$ Because $x$ is a nonwandering point, there is $t'>0$ such that $d(X_{t'}(x),x)<\delta$.  Let  $T_0,...,T_m \in \R$ be such that  
$$0=T_0<T_1<T_2<....<T_m=t'$$
 is a partition  for $[0,t']$ with $T_{i+1}-T_i \in [1,2]$. Let $p_0,...,p_m\in O(p)$ be such that 
$$d(p_j,X_{T_j}(x))< \delta' ~ \rm{for} ~ j=0,...,m.$$
We show that $X_{[0,t']}(x)$ is an $(\varepsilon,T,\eta)$-quasi hyperbolic arc. By using (\ref{equ3}) we have
$$
\begin{array}{l}
\frac{1}{T_k}\displaystyle\sum_{j=1}^k \text{log}\parallel \Psi_{T_j-T_{j-1}}\mid_{\Delta^s(X_{T_{j-1}}(x))}\parallel \leq \frac{1}{T_k}\displaystyle\sum_{j=1}^k \text{log}(\parallel \Psi'_{T_j-T_{j-1}}\mid_{\Delta_Y^s(p_j)}\parallel+\varepsilon')\\
\leq \frac{1}{T_k}\displaystyle\sum_{j=1}^k (\text{log}(\parallel \Psi'_{T_j-T_{j-1}}\mid_{\Delta_Y^s(p_j)}\parallel)+\varepsilon)\leq \frac{1}{T_k}\displaystyle\sum_{j=1}^k\text{log}(C'\lambda^{T_{j}-T_{j-1}})+\frac{k}{T_k}\varepsilon\\
\leq \frac{1}{T_k}\displaystyle\sum_{j=1}^k\text{log}(C'^{T_{j}-T_{j-1}}\lambda^{T_{j}-T_{j-1}})+\frac{k}{T_k}\varepsilon\leq  \displaystyle \text{log}(C')+\varepsilon = -\eta.
\end{array}
$$
For the first and second inequality, we used the properties in (11); for the third inequality, we used the hyperbolicty of $O(p)$; and for the fourth and fifth inequality, we used the property $T_{j}-T_{j-1}\geq 1$. 

\hspace{2mm}
On the other hand, we have $$m(\Psi'_{t}\mid_{\Delta_Y^u(y)})=\frac{1}{\parallel\Psi'_{-t}\mid_{\Delta_Y^{u}(Y_t(y))}\parallel}\geq C^{-1}\lambda^{-t}\geq {C'}^{-1}\lambda^{-t}. $$ 
Hence we get
$$
\begin{array}{l}
\frac{1}{T_m-T_{k-1}}\displaystyle\sum_{j=k}^m \text{log } m \Big( \Psi_{T_j-T_{j-1}}\mid_{\Delta^u(X_{T_{j-1}}(x))}\Big) \\ 
=\frac{1}{T_m-T_{k-1}}\displaystyle\sum_{j=1}^k \text{log}(\frac{1}{\parallel \Psi_{T_{j-1}-T_{j}}\mid_{\Delta^u(X_{T_j}(x))}\parallel}) \\ \geq 
 \frac{1}{T_m-T_{k-1}}\displaystyle\sum_{j=1}^k \text{log}(\frac{1}{\parallel \Psi'_{T_{j-1}-T_{j}}\mid_{\Delta_Y^u(p_j)}\parallel-\varepsilon'}) \\ 
\geq \frac{1}{T_m-T_{k-1}}\displaystyle\sum_{j=1}^k \Big(\text{log}(\frac{1}{\parallel \Psi'_{T_{j-1}-T_{j}}\mid_{\Delta_Y^u(p_j)}\parallel}))-\varepsilon\Big) \\ \geq
 \frac{1}{T_m-T_{k-1}}\displaystyle\sum_{j=k}^m \Big((T_{j-1}-T_j)(\text{log}(C')+ \text{log}(\lambda))\Big)-\frac{m-k+1}{T_m-T_{k-1}}\varepsilon \\ 
\geq -\text{log}(C')-\varepsilon - \text{log}(\lambda)\geq  -(\text{log}(C')+\varepsilon )=\eta.
\end{array}
$$
Similarly we obtain $$
\begin{array}{l}
\text{log} \parallel \Psi_{T_k-T_{k-1}}\mid_{\Delta^s(X_{T_{k-1}}(x))}\parallel-\text{log }m\Big(\Psi_{T_k-T_{k-1}}\mid_{\Delta^u(X_{T_{k-1}}(x)})\Big)\\ \leq \text{log} (C')+(T_k-T_{k-1})\text{log} (\lambda)+\varepsilon - \big(-\text{log} (C')+(-T_k+T_{k-1})\text{log} (\lambda)-\varepsilon\big)\\=
2\text{log }(C') + 2\varepsilon+2(T_k-T_{k-1})\text{log}(\lambda) \\\leq 2\text{log }(C') + 2\varepsilon=-2\eta,
\end{array}
$$
for all $k \in \{1,...,m\}$. Consequently we can see that $\Lambda$ contains a hyperbolic periodic orbit by Proposition \ref{quasi}.
\end{proof}


\begin{proof}[End of proof of main theorem]

Let $\Lambda$ be a chain transitive set, and suppose it is robustly shadowable. Then $\Lambda$ contains a hyperbolic periodic orbit, say $\gamma$,  by Lemma 4.1.  Since $\Lambda$ is transitive, we see that 
$\Lambda \subset C_{X}(\gamma)$ and also $\Lambda \subset H_{X}(\gamma)$. Since $\Lambda$ is compact and the periodic points are dense in $\Lambda$, we may assume that for any $T>0$ there is a periodic point $p$ in $\Lambda$ whose period is bigger than $T$. Then by using the results and techniques in Section 5 of \cite{LTW}, we can show that the dominated splitting $\mathcal{N}_{\Lambda} = \Delta^s \oplus \Delta^u$ is a hyperbolic spliting for $\Psi_t$. Consequently we can see that   $\Lambda$ is hyperbolic for $X_t$ by applying Proposition 2.5.

\hspace{2mm}
The converse is clear by the robust property of hyperbolic sets and the shadowability of the hyperbolic sets, and so completes the proof of our main theorem.
\end{proof}

\smallskip
{\it Acknowledgement.} The second author was supported by the NRF grant funded
by the Korea government (MSIP) (No. NRF-2015R1A2A2A01002437).


\begin{bibdiv}
\begin{biblist}

  
  \bib{AP}{book}{
   author={V. Ara\'ujo},
   author={M. J. Pacifivo},
   title={Three-dimensional flows},
   publisher={Springer, Berlin},
   date={1992},
 
}

   

		
 \bib{BGV}{article}{
   author={C. Bonatti},
   author={N. Gourmelon},
   author={T. Vivier},
   title={Perturbation of derivative along periodic orbits},
   journal={Ergod. Th. $\&$ Dynam. Syst},
   volume={26},
   date={2006},
   pages={1307--1337},
     
   
}


\bib{D}{article}{
   author={C. I. Doering},
   title={Persistently transitive vector fields on three- dimensional manifolds},
   journal={Dynam. Syst. $\&$ Bifurcat. Th.},
   volume={160},
   date={1987},
   pages={59--89},


}

\bib{GLT}{article}{
   author={S. Gan},
   author={M. Li},
   author={S. B. Tikhomirov},
   title={Oriented shadowing property and $\Omega$-stability for vector fields},
   journal={J. Dynamics and Differential Equaitons},
   volume={28},
   date={2016},
   pages={225--237},


   
   
  }
  
   \bib{GW}{article}{
   author={S. Gan},
   author={L. Wen},
   title={Nonsingular star flows satisfy Axiom A and the no-cycle condition},
   journal={Invent. math.},
   volume={164},
   date={2006},
   pages={279--315},
  
   }

 \bib{LS}{article}{
   author={K. Lee},
   author={K. Sakai},
   title={Structural stability of vector fields with shadowing},
   journal={J. Differential Equations},
   volume={232},
   date={2007},
   pages={303--313},
   
}

 \bib{LTW}{article}{
   author={K. Lee},
   author={L. H. Tien},
   author={X. Wen},
   title={Robustly shadowable chain components of $C^1$ vector fields},
   journal={J. Korean Math. Soc.},
   volume={51(1)},
   date={2014},
   pages={17--53},
  
}

\bib{LGW}{article}{
   author={M. Li},
   author={S. Gan},
   author={L. Wen},
   title={Robustly transitive singular set via approach of an extended linear \poincare flow},
   journal={Discrete Contin. Dyn. Syst.},
   volume={13},
   date={2005},
   pages={239--269},
   
}



 \bib{L}{article}{
   author={S. Liao},
   title={An existence theorem for periodic orbits},
   journal={Acta. Sci. Nat. Univ. Pekin.},
   volume={1},
   date={1979},
   pages={1--20},
   
   



}

\bib{PT}{article}{
   author={S. Y. Pilyugin},
   author={S. B. Tikhomirov},
   title={Vector fields with the oriented shadowing property},
   journal={J. Diff. Eqns.},
   volume={248},
   date={2010},
   pages={1345--1375},
  
  
}
\bib{PR}{article}{
   author={C. Pugh, C. Robinson},
   title={The $C^1$ closing lemma including Hamiltonians},
   journal={Ergod. Th. $\&$ Dynam. Syst.},
   volume={3},
   date={1983},
   pages={261--313},
  
}



\end{biblist}
\end{bibdiv}
\end{document}